\documentclass[10pt,a4paper,english]{article}
\usepackage[utf8]{inputenc}
\usepackage[english]{babel}
\usepackage{amsmath}
\usepackage{amsfonts}
\usepackage{amssymb}
\usepackage{amsthm}
\usepackage{pgfplots}
\usepackage{graphicx}
\usepackage{circuitikz}
\usepackage[justification=centering]{caption}
\usepackage[fixlanguage]{babelbib}
\graphicspath{ {./figures/} }
\usepackage{geometry}
\usepackage{placeins}
\usetikzlibrary{intersections, pgfplots.fillbetween}
\usetikzlibrary{patterns}
\usetikzlibrary{shapes,snakes}
\pgfplotsset{compat=1.16}
\newtheorem{Theorem}{Theorem}[section]
\newtheorem{Definitio}[Theorem]{Definition}
\newenvironment{Definition}{\begin{Definitio} \rm }{\end{Definitio}}
\newtheorem{Lemma}[Theorem]{Lemma}
\newtheorem{Proposition}[Theorem]{Proposition}
\newtheorem{Remar}[Theorem]{Remark}
\newenvironment{Remark}{\begin{Remar} \rm }{\end{Remar}}
\newtheorem{Exampl}[Theorem]{Example}

\newtheorem{Corollar}[Theorem]{Corollary}
\newenvironment{Corollary}{\begin{Corollar} \rm }{\end{Corollar}}
\author{Rudy Dissler}
\title{Multisections of $(m+3)$--dimensional $m$--spun $3$--manifolds}
\begin{document}
\maketitle
\begin{abstract}
A multisection, or $n$--section, of an $(n+1)$--dimensional manifold is a decomposition of this manifold into $n$ $1$--handlebodies of dimension $n+1$, such that all these handlebodies intersect along a closed surface, and every subcollection of $k$ handlebodies intersects along an $(n-k+2)$--dimensional $1$--handlebody. This concept, due to Ben Aribi, Courte, Golla and Moussard \cite{aribi2023multisections}, generalizes to any dimension Heegaard splittings and Gay--Kirby trisections. Any $(n+1)$--manifold admits a multisection if $n \leq 4$ (see \cite{gay2016trisecting} and \cite{aribi2023multisections}). But there are yet no general existence results for $n \geq 5$.  In this article, we provide a class of examples of multisected manifolds in all dimensions. We extend the concept of $4$--dimensional spun manifolds to any dimension, and construct multisections and their associated multisection diagrams for the class of $m$--spun $3$--manifolds, of dimension $m+3$, for any $m$. This allows us to give infinitely many examples of non-diffeomorphic multisected manifolds, in all dimensions.
\end{abstract}
\section{Introduction}
Throughout this article, we will refer to $1$--handlebodies (i.e. connected manifolds consisting in only $0$-- and $1$--handles) as handlebodies, and only consider closed, orientable manifolds. Decompositions of $3$--dimensional manifolds into two handlebodies intersecting along their common \mbox{boundary}, known as Heegaard splittings, have been intensively studied. This concept has been recently adapted to dimension $4$ by Gay and Kirby \cite{gay2016trisecting}, giving birth to the theory of trisected smooth $4$--manifolds. Then Rubinstein and Tillmann introduced in \cite{rubinstein2018generalized} a notion of multisection for PL--manifolds of any dimensions, which coincides with Heegaard splittings in dimension $3$ and trisections in dimension $4$. But their definition does not always allow a diagrammatic representation of a multisected manifold, as it is the case in dimension $3$ and $4$ (in particular because the intersection of all the handlebodies is not a surface if the dimension of the manifold is greater than~$5$).\\
We will consider in this article an alternative notion of multisection, from Ben Aribi, Courte, Golla and Moussard \cite{aribi2023multisections}, which also generalizes to any dimension the concept of Heegaard splitting and trisection and allows a diagrammatic approach. It is a well-known fact that any $3$--manifold admits a Heegaard splitting; the analogous statement is proven in dimension $4$ in \cite{gay2016trisecting} and in dimension $5$ in \cite{aribi2023multisections}. But it is yet not known if every smooth or PL manifold of dimension greater than or equal to six admits a multisection. This motivates the construction of examples of multisected manifolds in all dimensions (see \cite{aribi2023multisections} for basic examples and \cite{moussard2023multisections} for multisections of surface bundles and bundles over the circle). In this article, we will build multisections (or $n$--sections, $n$ referring to the number of pieces used in the decomposition, when the manifold is $(n+1)$--dimensional) and their associated multisection diagrams for a class of manifolds: what we call $m$--spun $3$--manifolds.\\
In the literature, spun manifolds are $4$--dimensional manifolds defined as follows. Consider a smooth $3$--manifold $M$, and denote by $M^o$ the compact $3$--manifold $M \setminus Int(B)$, where $B$ is a $3$--ball embedded in $M$. The spun of $M$ is the smooth, closed $4$--manifold $\partial\big(M^o \times B^{2} \big) = (M^{o} \times S^{1}) \cup_{S^{2} \times S^{1}} (S^{2} \times B^{2})$. We extend this definition to any dimension by defining the \mbox{$m$--spun} of a closed $n$--manifold $X$ as the $(n+m)$--dimensional manifold $\mathcal{S}_{m}(X) =  \partial (X^o \times B^{m+1}) = (X^{o} \times S^{m}) \cup_{S^{n-1} \times S^{m}} (S^{n-1} \times B^{m+1})$, where $X^o$ stands for $X$ minus the interior of an embedded $n$--ball. With these notations, the spun of a $3$--manifold $M$ as usually defined in dimension $4$ would be $\mathcal{S}_{1}(M)$, i.e. the $1$--spun of $M$. Notice that the definition of a spun manifold adapts naturally to the PL category.\\
In this article, we will consider $m$--spun $3$--manifolds, i.e. of the form $\mathcal{S}_{m}(M)=(M^o \times S^m) \cup_{S^2 \times S^{m}} (S^2 \times B^{m+1})$. Our goal is to provide multisections and their associated multisection diagrams for these spaces. We will generalize to any dimension the construction of Meier in \cite{MR3917737}, where he produces trisections and trisection diagrams for $1$--spun $3$--manifolds. We obtain the follo\-wing result, in the PL as well as in the smooth category, as a consequence of Proposition \ref{prop_multisection} and \mbox{Proposition \ref{prop_diagram}}.
\begin{Theorem}
\label{thm_introspun}
An $m$--spun $3$--manifold $\mathcal{S}_m(M)$ admits an $(m+2)$--section of genus $(m+2)g$, where $g$ is the genus of a Heegaard splitting of $M$. Moreover, an $(m+2)$--section diagram can be explicitly derived from a Heegaard diagram of $M$. 
\end{Theorem}  
By applying Theorem \ref{thm_introspun} to the family of lens spaces, which admit genus--$1$ Heegaard splittings, we obtain the following corollary (notice that $\mathcal{S}_{0}(M)$ is just the connected sum of $M$ with itself).
\begin{Corollary}
For any $n \geq 2$, there exists infinitely many non--diffeomorphic (resp. not PL--homeomorphic) $(n+1)$--dimensional smooth (resp. PL) manifolds admitting $n$--sections of genus~$n$. 
\end{Corollary}
First, we define multisections and their multisection diagrams. Then we construct the multisections for the $m$--spun $3$--manifolds. Finally, we give an algorithm to derive a multisection diagram from a Heegaard diagram of the fiber, and produce some examples: $m$--spun lens spaces, $m$--spun $3$--torus, and $m$--spun Poincaré homology sphere. 
\section{Multisections}
We follow \cite{aribi2023multisections}. We start by defining multisections in the PL category.
\begin{Definition}
\label{pl_multisection}
For $n \geq 1$, let $Y$ be an $(n+1)$--dimensional closed PL manifold. An $n$--section of $Y$ is a decomposition $Y= \bigcup_{i=1}^n Y_i$, where:
\begin{itemize}
\item the intersection $\cap_{i=1}^n Y_i$ of all the $Y_i$'s is a closed connected surface;
\item for any non-empty proper subset $I$ of $\lbrace 1,...,n \rbrace$, the intersection $Y_I = \cap_{i \in I} Y_i$ is PL--homeo\-morphic to an $(n-\lvert I \rvert +2)$--dimensional handlebody. 
\end{itemize} 
\end{Definition}
We will use handlebodies with corners in our definition of smooth multisections. To that end, we briefly define a submanifold with corners of a smooth manifold. 
\begin{Definition}
Let $Y$ be a smooth $n$--dimensional manifold, together with a smooth atlas $\mathcal{U}$. Let $W$ be a topological $m$--dimensional submanifold of $Y$. We say that $W$ is a submanifold with corners of $Y$ if for all $w \in W$, there exists a positive integer $k$ and a chart $(w \in U_w, \phi_w)$ in $\mathcal{U}$ centered at $w$, such that $\phi_w (U_w \cap W) = \mathbb{R}^{m}_k = [0, \infty)^k \times \mathbb{R}^{m-k}$. We say that such a $w$ belongs to the $(m-k)$--stratum of $W$. 
\end{Definition}
\begin{Remark}
Because there cannot exist a diffeomorphism of $\mathbb{R}^m$ sending $\mathbb{R}^{m}_k$ to $\mathbb{R}^{m}_{k'}$ if $k \neq k'$, each $w$ belongs to one and only one stratum of $W$. Therefore we have a partition of $W$ into its strata, which we call a \emph{stratification} of $W$.
\end{Remark}
\begin{Definition}
\label{def_multisection} 
For $n \geq 1$, let $Y$ be an $(n+1)$--dimensional closed smooth manifold. An $n$--section of $Y$ is a decomposition $Y= \bigcup_{i=1}^n Y_i$, such that, denoting by $Y_I$ the intersection $\cap_{i \in I} Y_i$:
\begin{itemize}
\item $Y_{\lbrace 1,...,n \rbrace}$ is a closed, connected smooth surface;
\item for $I \subset \lbrace 1,...,n \rbrace$ with $\lvert I \rvert \leq n-1$, $Y_I$ is a submanifold with corners of $Y$ with interior diffeomorphic to the interior of a smooth $(n - \lvert I \rvert +2)$--dimensional handlebody, and with a stratification given by $\lbrace \mathring{Y}_J, J \supseteq I \rbrace$, where $\mathring{Y}_J$ is in the $(n - \lvert J \rvert +2)$--stratum of $Y_I$ for all $J \supseteq I$. 
\end{itemize}
We call the intersection $Y_{\lbrace 1,...,n \rbrace}$ the \emph{central surface} of the multisection, its genus the \emph{genus of the multisection}, and one $Y_I$ a \emph{piece} of the multisection.
\end{Definition}
\begin{Remark}
\label{rmk_std}
Specifying the stratification of the $Y_I$'s fixes a decomposition of the normal bundle to each stratum. Consider the following decomposition of the $(k-1)$--simplex: $\Delta^{k-1}=\bigcup_{i=1}^{k} \Delta_{i}^{k-1}$, where $\Delta_{i}^{k-1}$ is the convex hull of $(p_0, p_1, ... , \hat{p}_i, ... , p_k)$, with $p_0$ the barycenter of $\Delta^{k-1}$ and $\hat{p}_i$ meaning we omit the term $p_i$. This decomposition naturally induces a decomposition of the euclidean space $\mathbb{R}^{k-1}= \bigcup_{i=1}^k A_i$, that we call $\emph{standard}$. Then, the condition on the $Y_I$'s means that, for $x \in \mathring{Y}_I$, there exists a neighborhood $U_x$ of $x$ in $Y$ and a diffeomorphism $\phi_x$ from $U_x$ to $\mathbb{R}^{n+1}$ that sends $U_x \cap Y_i$ to $\mathbb{R}^{n- \lvert I \rvert +2} \times A_i$ for all $i \in I$. We could imagine more relaxed definitions for the pieces (see Figure \ref{local_models}). In this case, however, we could obtain two multisections very similar, that is to say only differing by the angles between the pieces, but not diffeomorphic, in the sense that no diffeomorphism of the manifold could send one multisection to the other, piece by piece. For further details, see \cite{aribi2023multisections}.
\end{Remark}
\begin{figure}[h!]
\[
\begin{tikzpicture}[scale=0.3]
\node (mypic) at (0,0) {\includegraphics[scale=0.5]{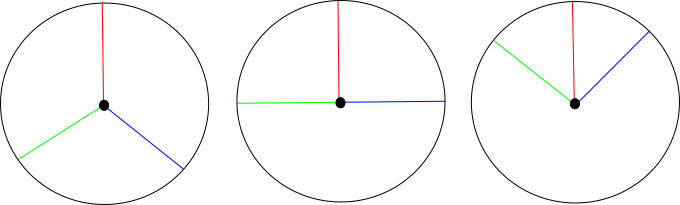}};
\end{tikzpicture}
\]
\caption{Non--diffeomorphic decompositions of the $2$--disk, with the decomposition corresponding to a standard local model on the left}
\label{local_models}
\end{figure}
\begin{Remark}
We can readily see from Definition \ref{def_multisection} that a multisection in dimension $3$ is a Heegaard splitting, and a multisection in dimension $4$ corresponds to a trisection, as usually defined. 
\end{Remark}
When constructing multisections of spun manifolds, we will first come up with a decomposition that matches all the criteria for a smooth multisection but one: the pieces are in fact handlebodies with smooth interior which are not submanifolds with corners, as their boundary presents some singularities not diffeomorphic to any type of corner. We now give a proper definition for such a decomposition. 
\begin{Definition}
\label{def_almost_smooth_multisection}
For $n \geq 1$, let $Y$ be an $(n+1)$--dimensional closed smooth manifold. An \emph{almost smooth} $n$--section of $Y$ is a decomposition $Y= \bigcup_{i=1}^n Y_i$, such that, denoting by $Y_I$ the intersection $\cap_{i \in I} Y_i$:
\begin{itemize}
\item $Y_{\lbrace 1,...,n \rbrace}$ is a closed smooth surface;
\item for any $I \subset \lbrace 1,...,n \rbrace$ with $\lvert I \rvert \leq n-1$, there exists an ambient isotopy $F_I$ of $Y$, smooth on each $\mathring{Y}_J$ for all $J \supseteq I$, such that $F_I(Y_I,1) \subset Y_I$ is a smooth $(n - \lvert I \rvert +2)$--dimensional handlebody.
\end{itemize}
\end{Definition}
\begin{Remark}
An almost smooth multisection induces a stratification of its pieces: $Y_I = \sqcup_{J \supseteq I} \mathring{Y}_J$, as well as an almost smooth multisection of the boundary of a smoothed piece: $\partial (F_I(Y_I,1)) = \cup_{J \supset I} F_I(Y_J,1)$. 
\end{Remark}
\begin{Remark}
\label{rmk_smoothing}
A smooth multisection is an almost smoothed multisection, as shown in \cite[Section~2]{aribi2023multisections}: for any $I$ with $\lvert I \rvert \leq n-1$, we can choose a smooth vector field in $Y$, transverse to each top--dimensional stratum of $\partial Y_I$. Then all the hypersurfaces in $Y_I$ (that is to say, the $(n- \lvert I \rvert+1)$--dimensional smooth submanifolds of $Y$ included in $\mathring{Y}_I$) that are homeomorphically mapped to $\partial Y_I$ under the flow of this vector field are canonically isotopic. This produces an isotopy from $Y_I$ to a smoothing of $Y_I$ sitting inside $Y_I$, and this isotopy is smooth on each $\mathring{Y}_J$ for all $J \supseteq I$.   
\end{Remark}
An interesting feature of multisections is that they can be described diagrammatically. 
\begin{Definition}
\label{cut_system}
Let $H$ be a 3–dimensional handlebody of genus $g$. A \emph{cut system} for $H$ is a family of $g$ disjoint, simple, homologically independent closed curves $\lbrace \alpha_{1},..., \alpha_{g} \rbrace$ on $\partial H$, such that each $\alpha_i$ bounds a disk in $H$. 
\end{Definition}
\begin{Definition}
\label{multisection_diagrams}
Let $Y=\bigcup_{i=1}^n Y_{i}$ be a smooth $n$--sected $(n+1)$--dimensional manifold. Denote by $\Sigma$ the central surface (of genus $g$) of the multisection. Choose, for $1 \leq i \leq n$, a cut system $\alpha^i = (\alpha_{j}^{i})_{1 \leq j \leq g}$ on $\Sigma$ for the $3$--dimensional handlebody $\cap_{k \neq i} Y_{k}$. Then $\lbrace \Sigma, \alpha^1,...,\alpha^n \rbrace$ is an \emph{\mbox{$n$--section} diagram} for the multisection. 
\end{Definition}
\begin{Remark}
\label{rmk_uniqueness}
A multisection diagram is not unique, but each system of curves $(\alpha_{j}^{i})$ is unique up to permutation and handleslides. Therefore the  diagram associated to a multisected manifold is unique up to diffeomorphism of the central surface and handleslides (performed independently within each system of curves).
\end{Remark}
There is a one-to-one correspondence between $n$--section diagrams (up to the transformations described in Remark \ref{rmk_uniqueness}) and $n$--sected PL manifolds (up to multisection preserving PL--homeomorphism), but the corresponding statement for smooth manifolds is true only for $n \leq 5$. For details, see \cite[Section~3]{aribi2023multisections}.
\section{Multisections of $m$--spun $3$--manifolds}
\label{secproof}
Let $M$ be a smooth $3$--manifold, together with a genus--$g$ Heegaard splitting $M=H_{\delta} \cup_{\Sigma} H_{\epsilon}$. Let $M^o$ be $M \setminus Int(B)$, where $B$ is an embedded $3$--ball inside of $M$. We assume that $B \subset Int(H_{\delta})$. We want to construct an $(m+2)$--section for $\mathcal{S}_{m}(M) = ( M^{o} \times S^{m}) \cup_{S^{2} \times S^{m}} (S^{2} \times B^{m+1})$. We obtain that $\mathcal{S}_{m}(M) =  \big((H_{\delta}^o \times S^{m}) \cup_{S^{2} \times S^{m}} (S^{2} \times B^{m+1}) \big) \cup_{\Sigma \times S^m} (H_{\epsilon} \times S^m) $, with $H_{\delta}^o = H_{\delta} \setminus Int (B)$.
\\
First, we decompose $B^{m+1}$ (see Figure \ref{fig_decomp_base}) into a collection of $m+2$ balls of dimension $m+1$, denoted by $(B_i)_{1 \leq i \leq m+2}$, such that:
\begin{itemize}
\item for $1\leq t \leq m+2$, the intersection of $t$ $B_i$'s is an $(m+2-t)$--ball;
\item each $B_i$ intersects $\partial B^{m+1}$ in an $m$--dimensional ball.
\end{itemize} 
Such a decomposition is obtained by projecting on $B^{m+1}$ the standard decomposition of the simplex $\Delta^{m+1}$. Considering $B^{m+1}$ as the union of its rays $r$ (i.e. segments from the origin to the boundary of $B^{m+1}$), we denote by $\partial^+ r$ the intersection of a ray $r$ with $\partial B^{m+1}=S^{m}$. Therefore $S^{m}= \cup_{r \subset B^{m+1}}\partial^+r$. In what follows, we denote by $N_r$ the copy of any subset $N$ of $M^o$ at $\partial^+ r$: $N_r = N \times \lbrace \partial^+ r \rbrace \subset M^o \times S^m$.
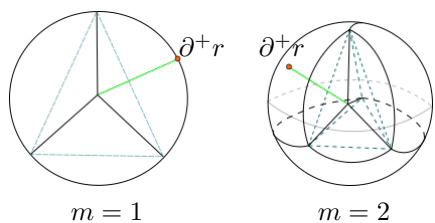
\begin{figure}[h!]
\[
\begin{tikzpicture}[scale=0.5]
\node (mypic) at (0,0) {\includegraphics[scale=0.5]{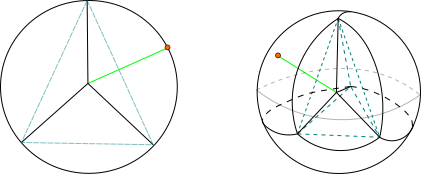}};
\node (a) at (-0.5,1.4) {$\partial^+ r$};
\node (a) at (1.6,1.4) {$\partial^+ r$};
\node (a) at (-3,-3) {$m=1$};
\node (a) at (3.5,-3) {$m=2$};
\end{tikzpicture}
\]
\caption{Decompositions of a $2$--disk and a $3$--ball, obtained from the standard decomposition of a simplex (one green ray $r$ is drawn on each)}
\label{fig_decomp_base}
\end{figure}
\\
Second, we decompose each $M^o_r = M^o \times \lbrace \partial^+r \rbrace$ into a $3$--dimensional handlebody and a collection of $3$--balls. To that extent, we consider $H_\delta^o$ as the union of a $2$--sphere times an interval and $g$ \mbox{$3$--dimensional} $1$--handles $(h_{\ell})_{1 \leq \ell \leq g}$. Then we define $(m+2)g$ $3$--dimensional tubes (solid cylinders) $h_\ell^i \subset h_{\ell}$, for $1 \leq i \leq m+2$ and $1 \leq \ell \leq g$, as follows. For all $1 \leq i \leq m+2$, we choose $g$ arcs in $H_\delta^o$, so that each arc is parallel to the core of one handle $h_{\ell}$, and is naturally extended to the spherical boundary component of $H_\delta^o$, using the product structure $2$--sphere times interval described above. Then one $h_\ell^i$ is defined as a regular neighborhood of such an arc in $H_\delta^o$, with the requirement that the resulting $(m+2)g$ tubes are pairwise disjoint (see Figure \ref{surger_handlebody}).
Notice that, for every $i$, the closure of $M^o$ minus the $g$ tubes $h_\ell^i$ is a regular neighborhood of $H_\epsilon$ in $M^o$.\\
Third, we notice that each tube $h_\ell^i$ intersects the spherical boundary component of $H_\delta^o$ in two disks. For a fixed ray $r$ of $B_i$, we can naturally extend our decomposition of $H_{\delta,r}^o$ to $H_{\delta,r}^o \cup (S^2 \times r)$, by defining a tube as the union of $h_\ell^i$ and these disks times $r$. We still denote by $h_{\ell,r}^i$ such a tube in $H_{\delta,r}^o \cup (S^2 \times r)$ (see Figure \ref{decomposition_fiber}). Then, fixing $i$, we define $\mathcal{H}_{\delta,r}^i$ as the closure of $H_{\delta,r}^o \cup (S^2 \times r)$ minus the tubes $h_{\ell,r}^i$. Finally, we set $\mathcal{H}_{\epsilon,r}^i = H_{\epsilon,r} \cup_{\Sigma_r} \mathcal{H}_{\delta,r}^i$. Then   $H_{\epsilon,r}$ is a deformation retract of $\mathcal{H}_{\epsilon,r}^i$.\\
As we want to keep notations as compact as possible, we denote by:
\begin{itemize}
\item $h_r^i$ the disjoint union of the $g$ tubes $h_{\ell,r}^i$;
\item $D_+ \sqcup D_-$ the projections of the positive and negative attaching regions of the handles $\lbrace h_{\ell} \rbrace_{1 \leq \ell \leq g}$ on the central copy of $S^2$: $S^2 \times \lbrace 0 \rbrace \subset S^2 \times B^{n+1}$, and by $D_{+}^i$ (resp. $D_{-}^i$) the disjoint union of the $g$ pairs of disks corresponding to the intersections of the $h_{\ell,r}^i$'s with $D_+$ (resp. $D_-$). See Figure \ref{decomposition_fiber}.
\end{itemize}

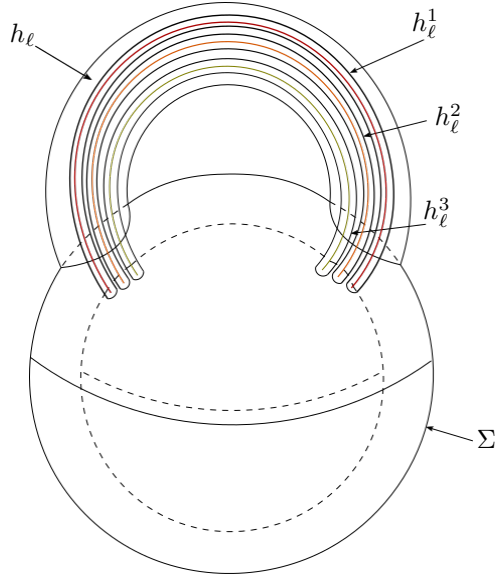
\begin{figure}[h!]
\[
\begin{tikzpicture}[scale=0.5]
\node (mypic) at (0,0) {\includegraphics[scale=0.5]{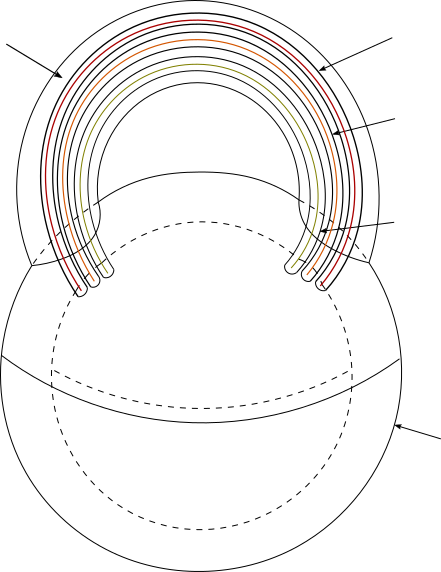}};
\node (a) at (6.2,-4) {$\Sigma$};
\node (a) at (5.2,4.5) {$h^{2}_{\ell}$};
\node (a) at (4.6,7) {$h^{1}_{\ell}$};
\node (a) at (4.9,2) {$h^{3}_{\ell}$};
\node (a) at (-6,6.7) {$h_{\ell}$};
\end{tikzpicture}
\]
\caption{Tubes $h^{i}_{\ell}$ in $H_{\delta}^o$ (only one of the $\ell$ handles is drawn and $m=1$)}
\label{surger_handlebody}
\end{figure}

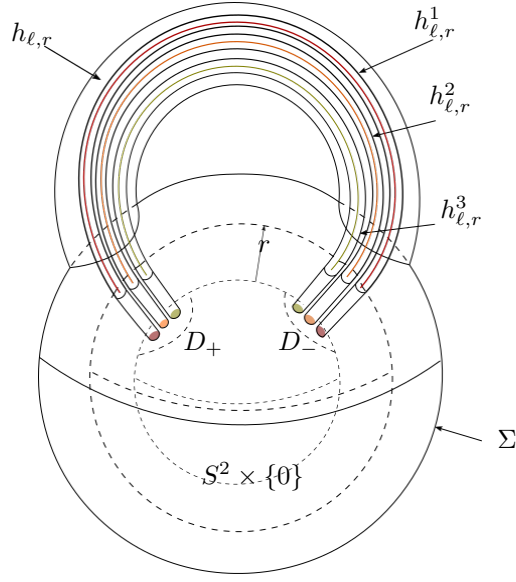
\begin{figure}[h!]
\[
\begin{tikzpicture}[scale=0.5]
\node (mypic) at (0,0) {\includegraphics[scale=0.5]{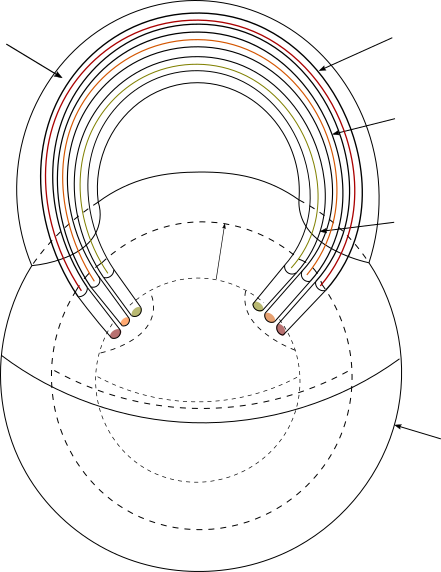}};
\node (a) at (0.1,1.1) {$r$};
\node (a) at (6.5,-4) {$\Sigma$};
\node (a) at (4.6,7) {$h_{\ell, r}^1$};
\node (a) at (5,5) {$h_{\ell, r}^2$};
\node (a) at (5.3,2) {$h_{\ell, r}^3$};
\node (a) at (-6,6.7) {$h_{\ell, r}$};
\node (a) at (-0.2,-5) {$S^2 \times \lbrace 0 \rbrace$};
\node (a) at (-1.5,-1.5) {$D_+$};
\node (a) at (1,-1.5) {$D_-$};
\end{tikzpicture}
\]
\caption{Extending the tubes inside $H_{\delta,r}^o \cup (S^2 \times r)$, with the coloured disks $D_{+}^i \subset D_+$ and $D_{-}^i \subset D_-$}
\label{decomposition_fiber}
\end{figure}
For $1 \leq i \leq m+2$ (considering indices modulo $(m+2)$), we define $Y_{i}= (\bigcup_{r \in B_i} \mathcal{H}_{\epsilon, r}^i) \cup (\bigcup_{r \in B_{i+1}}  h_{r}^{i+1})$. Therefore $\mathcal{S}_{m}(M)= \bigcup_{i=1}^{m+2} Y_{i}$. In what follows, for any collection of sets $\lbrace Z_1,...,Z_s \rbrace$ and any subset $K$ of $\lbrace 1,...,s\rbrace$, we denote by $Z_K$ the intersection $\cap_{k \in K} Z_k$. 

\begin{Proposition}
\label{prop_multisection}
The decomposition $\mathcal{S}_{m}(M)= \bigcup_{i=1}^{m+2} Y_{i}$ is an $(m+2)$--section of $\mathcal{S}_{m}(M)$, of genus $(m+2)g$. 
\begin{proof}
We want to show that $Y_{\lbrace 1,...,m+2 \rbrace}$ is a surface, and that $Y_K$ is a handlebody of dimension $m+4-\lvert K \rvert$ if $\lvert K \rvert \leq m+1$. It is convenient to start with sets of consecutive indices, and then move on to the general case. Note that $Y_K$ is diffeomorphic to $Y_{\lbrace 1,...k \rbrace}$ if $K$ is a set of $k$ consecutive indices.\\
First, we observe that $Y_i = (\bigcup_{r \in B_i} \tilde{M}_{r}^o \setminus h_{r}^i) \cup (\bigcup_{r \in B_{i+1}} h_{r}^{i+1})$, where $\tilde{M}_{r}^o = M_{r}^o \cup (S^2 \times r)$. 
Let $1 \leq k \leq m+2$. Using the fact that the $h_{r}^i$'s are pairwise disjoint, we decompose $Y_{\lbrace 1,...,k \rbrace}$ into three subsets:
\begin{enumerate}
\item $\mathcal{A}_{\lbrace 1,...k \rbrace} = \bigcup_{r \subset B_{\lbrace 1,...,k \rbrace}}  \overline{\tilde{M}_{r}^o \setminus (h_{r}^1 \cup ... \cup h_{r}^k})$ if $k \leq m+1$; $\mathcal{A}_{\lbrace 1,...,m+2 \rbrace}=\overline{S^2 \times \lbrace 0 \rbrace \setminus \cup_{i=1}^{m+2} D_{\pm}^i}$;
\item $\mathcal{B}_{\lbrace 1,...,k \rbrace} = \bigcup_{r \subset B_{\lbrace 1,...,k-1 \rbrace}} \overline{\tilde{M}_{r}^o \setminus (h_{r}^1 \cup ... \cup h_{r}^{k-1})} \cap \bigcup_{r \subset B_{k+1}} h_{r}^{k+1}$; notice that, for $k \leq m+1$, $\mathcal{B}_{\lbrace 1,...,k \rbrace} = \bigcup_{r \subset B_{\lbrace 1, ..., k-1,k+1 \rbrace}} h_{r}^{k+1}$, and $\mathcal{B}_{\lbrace 1,...,m+2 \rbrace}=\overline{\partial h_{r= B_{\lbrace 1,...,m+1 \rbrace}}^1 \setminus D_{\pm}^1}$;
\item $\mathcal{C}_{\lbrace 1,...,k \rbrace} = \bigcup_{s=1}^{k-1} \big( \bigcup_{r \subset B_{\lbrace 1, ... , \hat{s}, ... , k \rbrace}} \overline{\tilde{M}_{r}^o \setminus (h_{r}^1 \cup ... \cup h_{r}^k)} \cap \bigcup_{r \subset B_{s+1}}h_{r}^{s+1} \big)$, where $\hat{s}$ means we omit the index $s$. We can see that $\mathcal{C}_{\lbrace 1,...,k \rbrace}= \bigcup_{s=1}^{k-1} \bigcup_{r \subset B_{\lbrace 1, ... , \hat{s}, ... , k \rbrace}} \overline{\partial h_{r}^{s+1} \setminus D_{\pm}^{s+1}}$.
\end{enumerate}
We start with $k=m+2$. In this case, the union $\mathcal{B}_{\lbrace 1,...m+2 \rbrace} \cup \mathcal{C}_{\lbrace 1,...m+2 \rbrace}$ is equal to a collection of $(m+2)g$ disjoint cylinders: $\bigcup_{s=1}^{m+2} \bigcup_{r = B_{\lbrace 1, ... , \hat{s}, ... , m+2 \rbrace}} \overline{\partial h_{r}^{s+1} \setminus D_{\pm}^{s+1}}$. As each of these cylinders is glued by both extremities to $\mathcal{A}_{\lbrace 1, ... ,m+2 \rbrace}$ along the boundary of a pair of disks in some $D_\pm^i$, we obtain a surface of genus $(m+2)g$ (see Figure \ref{central_surface}).\\
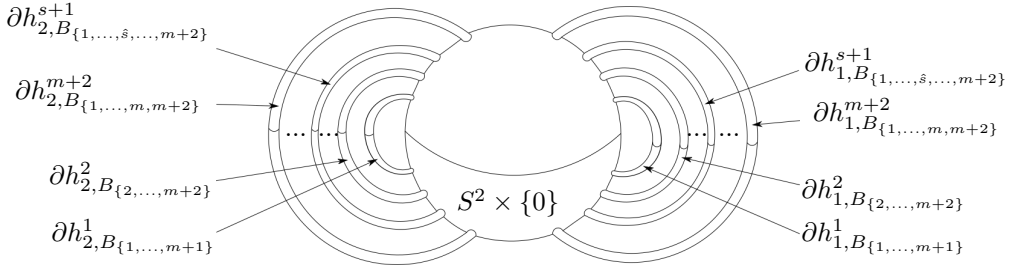
\begin{figure}[h!]
\[
\begin{tikzpicture}[scale=0.3]
\node (mypic) at (0,0) {\includegraphics[scale=0.3]{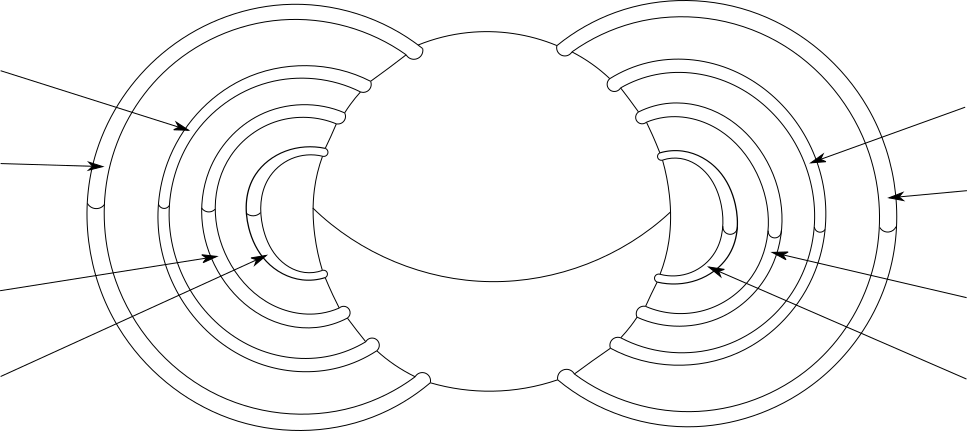}};
\node (a) at (16.5,-4.5) {$\partial h_{1,B_{\lbrace 1,...,m+1 \rbrace}}^1$};
\node (a) at (16.5,-2.6) {$\partial h_{1,B_{\lbrace 2,...,m+2 \rbrace}}^2$};
\node (a) at (17.5,3) {$\partial h_{1,B_{\lbrace 1,...,\hat{s},...,m+2 \rbrace}}^{s+1}$};
\node (a) at (17.5,0.8) {$\partial h_{1,B_{\lbrace 1,...,m,m+2 \rbrace}}^{m+2}$};
\node (a) at (-16.4,-4.5) {$\partial h_{2,B_{\lbrace 1,...,m+1 \rbrace}}^1$};
\node (a) at (-16.5,-2) {$\partial h_{2,B_{\lbrace 2,...,m+2 \rbrace}}^2$};
\node (a) at (-17.5,5) {$\partial h_{2,B_{\lbrace 1,...,\hat{s},...,m+2 \rbrace}}^{s+1}$};
\node (a) at (-17.5,1.8) {$\partial h_{2,B_{\lbrace 1,...,m,m+2 \rbrace}}^{m+2}$};
\node (a) at (8.29,0) {...};
\node (a) at (-7.85,0) {...};
\node (a) at (9.7,0) {...};
\node (a) at (-9.3,0) {...};
\node (a) at (0,-3) {$S^2 \times \lbrace 0 \rbrace$};
\end{tikzpicture}
\]
\caption{Central surface of the multisection, when $g=2$}
\label{central_surface}
\end{figure}
Then, if $k \leq m+1$, for a ray $r$ in $B_{\lbrace 1,...,k \rbrace}$, $\overline{\tilde{M}_{r}^o \setminus (h_{r}^1 \cup ... \cup h_{r}^k)}$ is equal to $\mathcal{H}_{\epsilon,r}$ minus the interior of $(k-1)g$ boundary parallel tubes of dimension $3$, therefore it is a $3$--dimensional handlebody $\mathcal{H}$, of genus $gk$. Then $\mathcal{A}_{\lbrace 1,...,k \rbrace}$ is diffeomorphic to $(\mathcal{H} \times B^{m-k+1}) / \sim$, where $\sim$ means that the central copies of $S^2$ minus the interior of the attaching spheres of the handles are identified for all $r$. Therefore $\mathcal{A}_{\lbrace 1,...,k \rbrace}$ is a genus $gk$, $(m-k+4)$--dimensional handlebody (with part of a collar neighbourhood of its boundary collapsed). Then we deal with $\mathcal{B}_{\lbrace 1,...,k \rbrace}$: it is a collection of $g$ disjoint $(m-k+4)$--dimensional balls, glued to $\mathcal{A}_{\lbrace 1,...,k \rbrace}$ along $\bigcup_{r \subset B_{\lbrace 1, ...,k+1 \rbrace}} h_{r}^{k+1}$. Therefore, if $k=m+1$, it is a collection of $g$ $3$--dimensional balls glued to $\mathcal{A}_{\lbrace 1,...,m+1 \rbrace}$ along two disks of their boundary, and we obtain that $\mathcal{A}_{\lbrace 1,...,m+1 \rbrace} \cup \mathcal{B}_{\lbrace 1,...,m+1 \rbrace}$ is a $3$--dimensional handlebody of genus $g(m+2)$. If $k \leq m$, each component of $\mathcal{B}_{\lbrace 1,...,k \rbrace}$ is an $(m-k+4)$--dimensional ball glued to $\mathcal{A}_{\lbrace 1,...,k \rbrace}$ along an $(m-k+3)$--dimensional ball in its boundary, therefore $\mathcal{A}_{\lbrace 1,...,k \rbrace} \cup \mathcal{B}_{\lbrace 1,...,k \rbrace} \simeq \mathcal{A}_{\lbrace 1,...,k \rbrace}$. As for $\mathcal{C}_{\lbrace 1,...,k \rbrace}$, it is just a thickening of the handles of $\mathcal{A}_{\lbrace 1,...,k \rbrace}$ and $\mathcal{B}_{\lbrace 1,...,k \rbrace}$.\\ 
Now we deal with the general case.\\ 
If $\lvert K \rvert =m+1$, we remove just one $Y_{i}$ from $\lbrace Y_{1},...,Y_{m+2} \rbrace$: any intersection of $(m+1)$ $Y_i$'s is diffeomorphic to $Y_{\lbrace 1,...,m+1 \rbrace}$. Therefore this intersection is always a $3$--dimensional handlebody of genus $g(m+2)$.\\
If $\lvert K \rvert \leq m$, we already showed that, for any subset of $K' \subset K$ constituted of consecutive indices, $Y_{K'}$ is an $(m-\lvert K' \rvert+4)$--dimensional handlebody of genus $g \lvert K' \rvert$. We can always split the set of indices $K$ between disjoint subsets of consecutive indices: $K=\lbrace K_1,...,K_s \rbrace$, with $max(K_i) < min(K_{i+1})$. Using again the fact that tubes of different indices are disjoint, we obtain a decomposition of $Y_K$ into three subsets:
\begin{enumerate}
\item $\mathcal{A}_K=\bigcup_{r \subset B_{K}}  \overline{\tilde{M}_{r}^o \setminus \cup_{k \in K} h_r^k}$, an $(m-\lvert K \rvert +4)$--dimensional handlebody of genus $g \lvert K \rvert$;
\item $\mathcal{B}_K = \bigcup _{1 \leq t \leq s} \bigcup_{r \subset B_{\tilde{K}}} h_r^{max(K_t)+1}$, where $\tilde{K} = (K \cup \lbrace max(K_t) +1 \rbrace ) \setminus \lbrace max(K_t) \rbrace$; it is a disjoint union of $(m-\lvert K \rvert +4)$--balls glued to $\mathcal{A}_K$ along an $(m-\lvert K \rvert +3)$--ball in their boundary, therefore $\mathcal{A}_K \cup \mathcal{B}_K \simeq \mathcal{A}_K$;
\item $\mathcal{C}_K = \bigcup_{k \in K} \bigcup_{ j \in K_k \setminus \lbrace max(K_k) \rbrace} \bigcup_{r \subset B_{K \setminus \lbrace j \rbrace}}  \overline{\partial h_{r}^{j+1} \setminus D_{\pm}^{j+1}}$, which only thickens some of the hand\-les of $\mathcal{A}_K$.
\end{enumerate}
Therefore, if $\lvert K \rvert \leq m$, $Y_{K}$ is an $(m-\lvert K \rvert +4)$--dimensional handlebody of genus $g \lvert K \rvert$.\\
We have thus proven that, in the PL setting, $\mathcal{S}_{m}(M)= \bigcup_{i=1}^{m+2} Y_{i}$ is an $(m+2)$--section of genus $(m+2) g$ of the $m$--spun of $M$. In the smooth setting, we need to slightly modify the construction. First, we need to work with a \emph{smooth} central surface. To address this point, we modify the decomposition of $H_{\delta,r}^o$: instead of simply extending each handle $h_{\ell,r}^i$ with a solid cylinder, we can use a smooth function to isotope the disks inside the product structure, and obtain a smooth surface. This does not change the general construction. Therefore the $Y_K$'s have smooth interior, although it does not mean that they are manifolds with corners. In fact, their boundary presents a singularity on their intersection with $\textbf{0} \times S^2$: around the central copy of $S^2$, the decomposition of the normal bundle to a $Y_K$ induced by the multisection switches gradually from a standard decomposition to a decomposition with one flat angle, as represented on Figure \ref{models_along_central_surface} (this flat angle is given by two collinear vectors with opposite directions, one pointing inside a tridimensional tube $h_{\ell,r}^i$, the other pointing inside a tridimensional handlebody $\mathcal{H}_{\epsilon, r}^i$, or equivalently, outside the tube $h_{\ell,r}^i$). 
\begin{figure}[h!]
\[
\begin{tikzpicture}[scale=0.5]
\node (mypic) at (0,0) {\includegraphics[scale=0.5]{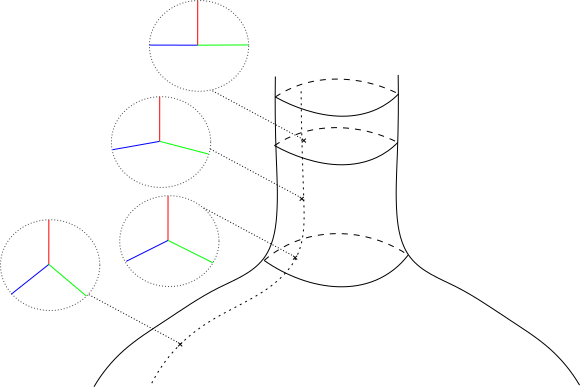}};
\node (a) at (6.2,-5) {$\Sigma$};
\node (a) at (-7,-1.6) {$Y_1$};
\node (a) at (-5.8,-1.6) {$Y_2$};
\node (a) at (-6.4,-2.6) {$Y_3$};

\node (a) at (-3.8,-1) {$Y_1$};
\node (a) at (-2.6,-1) {$Y_2$};
\node (a) at (-3.2,-2) {$Y_3$};

\node (a) at (-4,1.7) {$Y_1$};
\node (a) at (-2.8,1.7) {$Y_2$};
\node (a) at (-3.4,0.6) {$Y_3$};

\node (a) at (-3,4.3) {$Y_1$};
\node (a) at (-1.8,4.3) {$Y_2$};
\node (a) at (-2.4,3.2) {$Y_3$};
\end{tikzpicture}
\]
\caption{Decomposition of the normal bundle to the central surface $\Sigma$, induced by the almost smooth trisection of $\mathcal{S}_{1}(M)$}
\label{models_along_central_surface}
\end{figure}
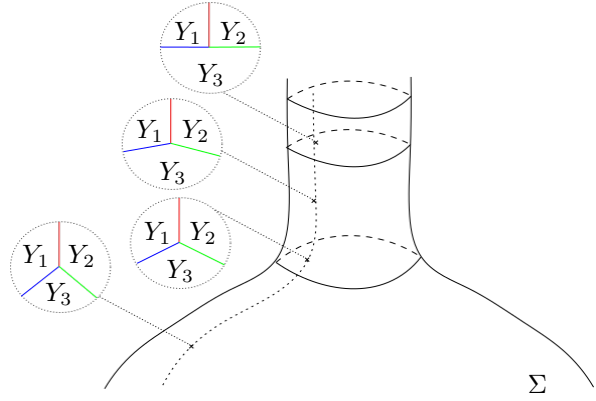
However, our decomposition induces, for any $K$, a stratification $\partial W_K = \sqcup_{1 \in K} \mathring{W}_{K}$, and, as in Remark \ref{rmk_smoothing}, we can choose a smooth vector field $\chi$ on $\mathcal{S}_{m}(M)$, transverse to each top dimensional stratum of $\partial W_K$. Then all hypersurfaces embedded in the interior of $W_K$ which are homeomorphically mapped to $\partial W_K$ under the flow of $\chi$ are canonically isotopic (and inherit an almost smooth $(n - 1)$--section from the decomposition of $\partial W_1$). This defines an isotopy from $\partial W_K$ to any of these hypersurfaces, which is smooth on each $\mathring{Y}_J$ for $J \supseteq K$. Therefore we have produced an almost smooth multisection of $\mathcal{S}_{m}(M)$, from which we now derive a smooth multisection. Note that the normal bundle to the interior of a piece is trivial (see Proposition \ref{trivial_bundle} below). We start by modifying our decomposition on a neighborhood of the central surface. This amounts to a cut and paste operation (see Figure \ref{modifying_local_model} for an example in dimension $4$), that adds or removes a small ball along the boundary of each $Y_K$. 
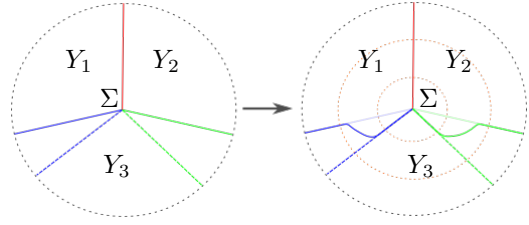
\begin{figure}[h!]
\[
\begin{tikzpicture}[scale=0.5]
\node (mypic) at (0,0) {\includegraphics[scale=0.5]{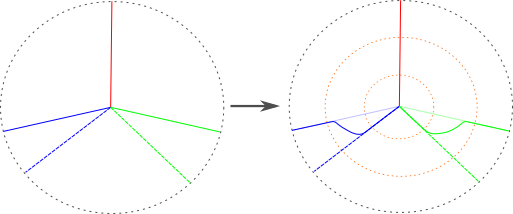}};
\node (a) at (-5,1.3) {$Y_1$};
\node (a) at (-2.7,1.3) {$Y_2$};
\node (a) at (-4,-1.5) {$Y_3$};
\node (a) at (-4.2,0.3) {$\Sigma$};

\node (a) at (2.7,1.3) {$Y_1$};
\node (a) at (5,1.3) {$Y_2$};
\node (a) at (4,-1.5) {$Y_3$};
\node (a) at (4.2,0.3) {$\Sigma$};

\end{tikzpicture}
\]
\caption{Modifying the decomposition of the normal bundle to the central surface $\Sigma$, induced by the almost smooth trisection of $\mathcal{S}_{1}(M)$}
\label{modifying_local_model}
\end{figure}
We obtain a new almost smooth multisection of $\mathcal{S}_{m}(M)$, that admits a standard decomposition of the normal bundle to the central surface. Next we can perform the same operation on a neighborhood of the tridimensional pieces, extending on this neighborhood the standard decomposition induced by the standard decomposition of the normal bundle to the central surface. We can thus work inductively until all the decompositions of the normal bundles to the $Y_K$'s induced by the multisection are standard. We obtain a smooth multisection of $\mathcal{S}_{m}(M)$, derived from the initial almost smooth multisection only by adding or removing balls along the boundaries of the $Y_K$'s.       
\end{proof} 
\end{Proposition}
Now we show that the normal bundle to each piece of an almost smoothly multisected manifold $W$ is trivial, and therefore that each piece admits a product neighborhood in $W$. In what follows, we denote by $\nu^W N$ the normal bundle in $W$ to a smooth submanifold $N \subset W$. We rely on the following classical lemma (see for instance \cite[chapter~3]{kosinski2013differential}).
\begin{Lemma}
\label{bundle_lemma}
Let $N \subset N' \subset W$, with $N$, $N'$ and $W$ smooth. Then $\nu^W N = \nu^{N'} N \oplus \nu^W N'_{\lvert N}$. 
\end{Lemma}
\begin{Proposition}
\label{trivial_bundle}
Let $W=\bigcup_{i=1}^n W_i$ be an almost smoothly $n$--sected $(n+1)$--manifold. Then for all $K \subset \lbrace 1,...,n \rbrace$, the normal bundle to $\mathring{W_K}$ in $W$ is trivial.
\end{Proposition}
\begin{proof}
Without loss of generality, we consider that $K=\lbrace 1,...,k \rbrace$. We prove the proposition by induction on $n$. To apply Lemma \ref{bundle_lemma}, we need to work with smooth submanifolds of $W$. So consider $\mathring{W_{K}} \subset W_1$: by definition of an almost smooth multisection, there exists an ambient isotopy $F_1$ of $W$ such that $F_1 (W_1,1) \subset W_1$ is a smooth handlebody, that we will denote by $W_{1}^s$. Moreover, $\partial W_{1}^s$ inherits an almost smooth $(n - 1)$--section from the decomposition of $\partial W_1$: $\partial W_1^s= \cup_{j \geq 2} W_{1,j}^s$, where $W_{1,j}^s = F_1 (W_{1,j})$. As, by definition, $F_1$ is continuous on $\partial W_1$ and smooth on each $\mathring{W_J} \subset \partial W_1$, $\mathring W_K$ is smoothly isotopic to $\mathring{W_K^s}$ in $W$, therefore (\cite[chapter~3]{kosinski2013differential}, for instance), the normal bundle of $\mathring{W_K}$ in $W$ is isomorphic to the normal bundle of $\mathring{W_K^s}$ in $W$.    
Now we can prove the proposition by induction on $n$. For $n=1$, the proposition is trivial. If the proposition is true for any $n$--sected smooth $(n+1)$--manifold, then consider a smooth $(n+1)$--sected $(n+2)$--manifold and $1 \leq k \leq n+1$. Using Lemma \ref{bundle_lemma}, we have: $\nu^W \mathring{W_K} \simeq \nu^W \mathring{W_K^s} = \nu^W {W_1^s}_{\lvert \mathring{W_K^s}} \oplus \nu^{W_1^s} \mathring{W_K^s}= \nu^{W_1^s} \mathring{W_K^s} = \nu^{W_1^s} {\partial W_1^s}_{\lvert \mathring{W_K^s}} \oplus \nu^{\partial W_1^s}\mathring{W_K^s}$. But $\nu^{W_1^s} {\partial W_1^s}_{\lvert \mathring{W_K^s}}$ is trivial because a boundary admits a collar neighborhood, and $\nu^{\partial W_1^s}\mathring{W_K^s}$ is trivial by induction.
\end{proof}
\begin{Remark}
As a smooth multisection is an almost smooth multisection (see Remark \ref{rmk_smoothing}), we have shown that the normal bundle to the interior of a piece of a smooth multisection is trivial.
\end{Remark}
\begin{Remark}
The proof of Proposition \ref{prop_multisection} allows to compute the genera of the handlebodies $Y_K$ involved, which are: $(m+2)g$ for $\lvert K \rvert = m+1$ ($3$--dimensional intersections), and $\lvert K \rvert g$ for $m \geq \lvert K \rvert \geq 1$, which shows that the decomposition is \emph{balanced} (if $\lvert K \rvert = \lvert I \rvert$, $Y_K$ and $Y_I$ have same genus).
\end{Remark}
\section{Associated multisection diagrams} 
\label{section_diagrams}
To construct the multisection diagrams associated to the decompositions of Section \ref{secproof}, we have to produce, for every $i$, a cut system $\alpha^i$ of the central surface corresponding to $\bigcap_{j \neq i} Y_{j}$. Equivalently, we need to find, for every $\bigcap_{j \neq i} Y_{j}$, a system of $(m+2)g$ curves on the central surface that bound disks in $\bigcap_{j \neq i} Y_{j}$, and such that surgering the central surface along these curves produces $S^2$. We adapt the process described in \cite{MR3917737} to any dimension. We start from a Heegaard diagram for $M$, denoted by $\lbrace \Sigma, \delta, \epsilon \rbrace$, such that $H_{\delta}$ is \emph{standardly embedded} in $S^3$: Figure \ref{standard_diagram} shows an embedding of $\Sigma$ in $S^3$, together with two sets of curves $\mu$ and $\lambda$; the Heegaard diagram $\lbrace \Sigma, \delta, \epsilon \rbrace$ defines a handlebody $H_{\delta}$ standardly embedded in $S^3$ if, with $\Sigma$ embedded as in Figure \ref{standard_diagram}, either $\delta = \mu$ or $\delta = \lambda$. It is a well-known fact that a $3$--manifold always admits such a Heegaard diagram.
\begin{figure}[h!]
\[
\begin{tikzpicture}[scale=0.5]
\node (mypic) at (0,0) {\includegraphics[scale=0.5]{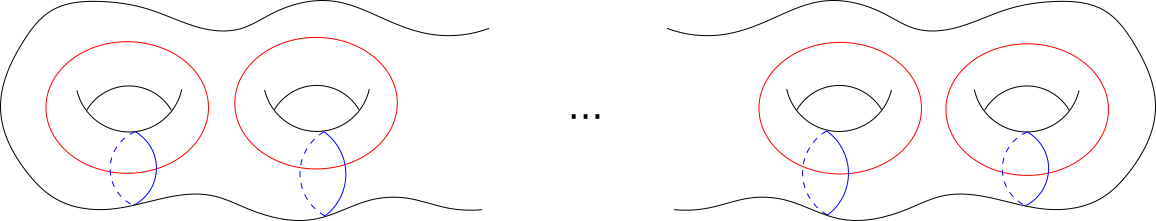}};
\end{tikzpicture}
\]
\caption{The cut systems $\mu$ (in blue) and $\lambda$ (in red) on the Heegaard surface $\Sigma$}
\label{standard_diagram}
\end{figure}
\begin{Proposition}
\label{prop_diagram}
Consider a genus--$g$ Heegaard diagram $\lbrace \Sigma, \delta, \epsilon \rbrace$ for $M$, such that $H_{\delta}$ is standardly embedded in $S^3$, with the $\delta$--curves system equal to the $\lambda$--curves system in Figure \ref{standard_diagram}. A multisection diagram for $\mathcal{S}_m(M)$ can be obtained as follows.
\begin{itemize}
\item Consider an embedding of the central surface as in Figure \ref{model_surface}, where two sets of curves $\lbrace \lambda_j^k \rbrace_{1 \leq j \leq m+2}^{1 \leq k \leq g}$ and $\lbrace \mu_j^k \rbrace_{1 \leq j \leq m+2}^{1 \leq k \leq g}$ are defined.
\item For each $1 \leq i \leq m$, construct the $\alpha^i$--curves in two steps (the indices $j \in \lbrace 1,...,m+2 \rbrace$ are considered modulo $m+2$).
\begin{enumerate}
\item The first $mg$ curves are the $\lbrace \mu_j^k \rbrace_{i \leq j \leq m+i-1}^{1 \leq k \leq g}$, the next $g$ curves are the $\lbrace \lambda_{i+m+1}^k \rbrace_{1 \leq k \leq g}$ (see Figure \ref{model_surface} and \ref{first_step_example}). 
\item Identifying the $\lambda$--curves of the Heegaard diagram with the $g$ curves obtained by sliding $\lambda_{i}^k$ over each of the $\lambda_{j}^k$, for $i+1 \leq j \leq m+i$ and every $1 \leq k \leq g$, defines a homeomorphism from $\Sigma$ to the central surface surgered along the first $(m+1)g$ $\alpha^i$--curves. The last $g$ $\alpha^i$--curves are the images of the $\epsilon$--curves by this homeomorphism, isotoped on the central surface so that they are disjoint from the curves previously drawn. 
\end{enumerate}
\end{itemize}
If the $\delta$--curves system is equal to the $\mu$--curves system in Figure \ref{standard_diagram}, proceed in the same way, only changing step $1$ as follows: \emph{the first $mg$ curves are the $\lbrace \lambda_j^k \rbrace_{i \leq j \leq m+i-1}^{1 \leq k \leq g}$, the next $g$ curves are the $\lbrace \mu_{m+i}^k \rbrace_{1 \leq k \leq g}$}, and adapting the second step accordingly.
\end{Proposition}
See Figure \ref{first_step_examples} and Figure \ref{multisection_diagram_s2l21} for examples.
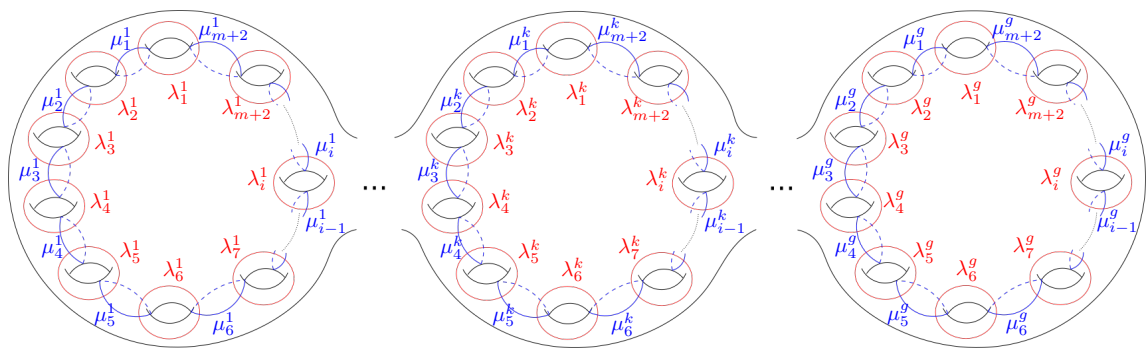
\begin{figure}[h!]
\[
\begin{tikzpicture}[scale=0.3]
\node (mypic) at (0,0) {\includegraphics[scale=0.33]{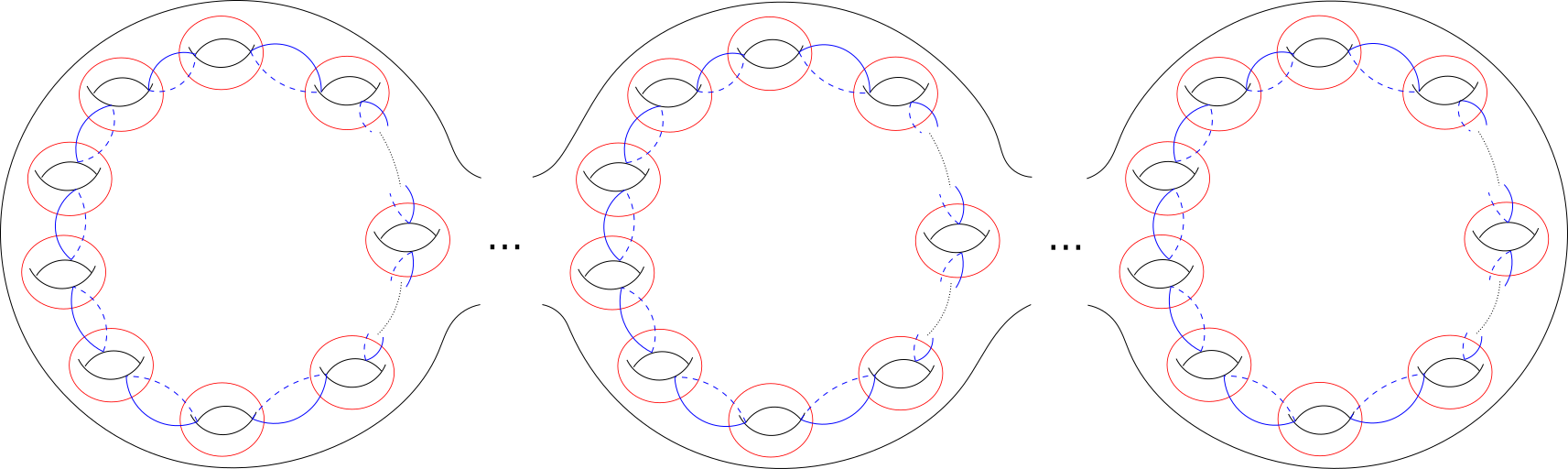}};
\node[red, scale=0.8] (a) at (-14,0) {$\lambda_{i}^1$};
\node[red, scale=0.8] (a) at (-14.5,3.1) {$\lambda_{m+2}^1$};
\node[red, scale=0.8](a) at (-17.5,3.8) {$\lambda_{1}^1$};
\node[red, scale=0.8] (a) at (-19.7,3.1) {$\lambda_2^1$};
\node[red, scale=0.8] (a) at (-20.7,1.6) {$\lambda_3^1$};
\node[red, scale=0.8](a) at (-20.9,-1) {$\lambda_4^1$};
\node[red, scale=0.8] (a) at (-19.6,-3.2) {$\lambda_5^1$};
\node[red, scale=0.8] (a) at (-17.7,-4) {$\lambda_6^1$};
\node[red, scale=0.8] (a) at (-15.2,-3) {$\lambda_7^1$};
\node[blue, scale=0.8] (a) at (-15.5,-6.4) {$\mu_6^1$};
\node[blue, scale=0.8] (a) at (-20.7,-6) {$\mu_5^1$};
\node[blue, scale=0.8] (a) at (-23,-3) {$\mu_4^1$};
\node[blue, scale=0.8] (a) at (-24,0.4) {$\mu_3^1$};
\node[blue, scale=0.8] (a) at (-23,3.5) {$\mu_2^1$};
\node[blue, scale=0.8] (a) at (-20,6.2) {$\mu_1^1$};
\node[blue, scale=0.8] (a) at (-15.5,6.6) {$\mu_{m+2}^1$};
\node[blue, scale=0.8] (a) at (-11,1.5) {$\mu_{i}^1$};
\node[blue, scale=0.8] (a) at (-11,-1.9) {$\mu_{i-1}^1$};
\begin{scope}[xshift=17.5cm]
\node[red, scale=0.8] (a) at (-14,0) {$\lambda_{i}^k$};
\node[red, scale=0.8] (a) at (-14.5,3.1) {$\lambda_{m+2}^k$};
\node[red, scale=0.8](a) at (-17.5,3.8) {$\lambda_{1}^k$};
\node[red, scale=0.8] (a) at (-19.7,3.1) {$\lambda_2^k$};
\node[red, scale=0.8] (a) at (-20.7,1.6) {$\lambda_3^k$};
\node[red, scale=0.8](a) at (-20.9,-1) {$\lambda_4^k$};
\node[red, scale=0.8] (a) at (-19.6,-3.2) {$\lambda_5^k$};
\node[red, scale=0.8] (a) at (-17.7,-4) {$\lambda_6^k$};
\node[red, scale=0.8] (a) at (-15.2,-3) {$\lambda_7^k$};
\node[blue, scale=0.8] (a) at (-15.5,-6.4) {$\mu_6^k$};
\node[blue, scale=0.8] (a) at (-20.7,-6) {$\mu_5^k$};
\node[blue, scale=0.8] (a) at (-23,-3) {$\mu_4^k$};
\node[blue, scale=0.8] (a) at (-24,0.4) {$\mu_3^k$};
\node[blue, scale=0.8] (a) at (-23,3.5) {$\mu_2^k$};
\node[blue, scale=0.8] (a) at (-20,6.2) {$\mu_1^k$};
\node[blue, scale=0.8] (a) at (-15.5,6.6) {$\mu_{m+2}^k$};
\node[blue, scale=0.8] (a) at (-11,1.5) {$\mu_{i}^k$};
\node[blue, scale=0.8] (a) at (-11,-1.9) {$\mu_{i-1}^k$};
\end{scope}
\begin{scope}[xshift=34.8cm]
\node[red, scale=0.8] (a) at (-14,0) {$\lambda_{i}^g$};
\node[red, scale=0.8] (a) at (-14.5,3.1) {$\lambda_{m+2}^g$};
\node[red, scale=0.8](a) at (-17.5,3.8) {$\lambda_{1}^g$};
\node[red, scale=0.8] (a) at (-19.7,3.1) {$\lambda_2^g$};
\node[red, scale=0.8] (a) at (-20.7,1.6) {$\lambda_3^g$};
\node[red, scale=0.8](a) at (-20.9,-1) {$\lambda_4^g$};
\node[red, scale=0.8] (a) at (-19.6,-3.2) {$\lambda_5^g$};
\node[red, scale=0.8] (a) at (-17.7,-4) {$\lambda_6^g$};
\node[red, scale=0.8] (a) at (-15.2,-3) {$\lambda_7^g$};
\node[blue, scale=0.8] (a) at (-15.5,-6.4) {$\mu_6^g$};
\node[blue, scale=0.8] (a) at (-20.7,-6) {$\mu_5^g$};
\node[blue, scale=0.8] (a) at (-23,-3) {$\mu_4^g$};
\node[blue, scale=0.8] (a) at (-24,0.4) {$\mu_3^g$};
\node[blue, scale=0.8] (a) at (-23,3.5) {$\mu_2^g$};
\node[blue, scale=0.8] (a) at (-20,6.2) {$\mu_1^g$};
\node[blue, scale=0.8] (a) at (-15.5,6.6) {$\mu_{m+2}^g$};
\node[blue, scale=0.8] (a) at (-11,1.5) {$\mu_{i}^g$};
\node[blue, scale=0.8] (a) at (-11.1,-1.9) {$\mu_{i-1}^g$};
\end{scope}
\end{tikzpicture}
\]
\caption{Embedding of the central surface featuring the curves used in the first step of the algorithm}
\label{model_surface}
\end{figure}

\begin{figure}[h!]
\[
\begin{tikzpicture}[scale=0.3]
\node (mypic) at (0,0) {\includegraphics[scale=0.3]{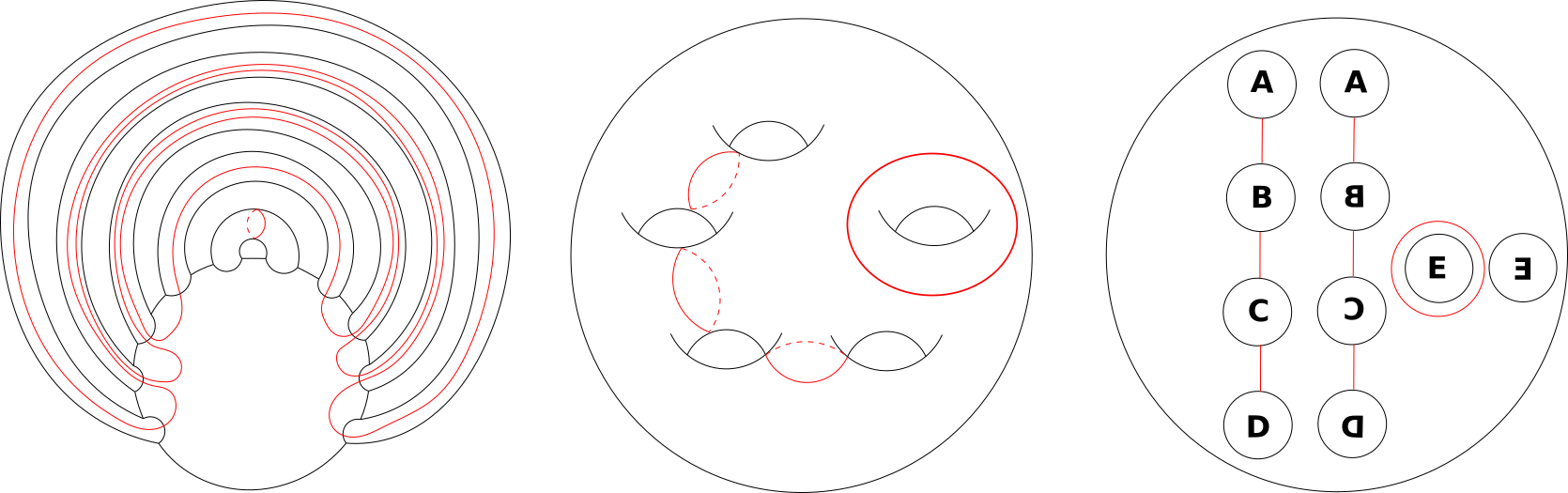}};
\end{tikzpicture}
\]
\caption{Various embeddings of the central surface when $m=3$ and $g=1$, showing the curves $\alpha^1$ obtained during the first step}
\label{first_step_example}
\end{figure}
\begin{figure}[h!]
\[
\begin{tikzpicture}[scale=0.3]
\node (mypic) at (0,0) {\includegraphics[scale=0.3]{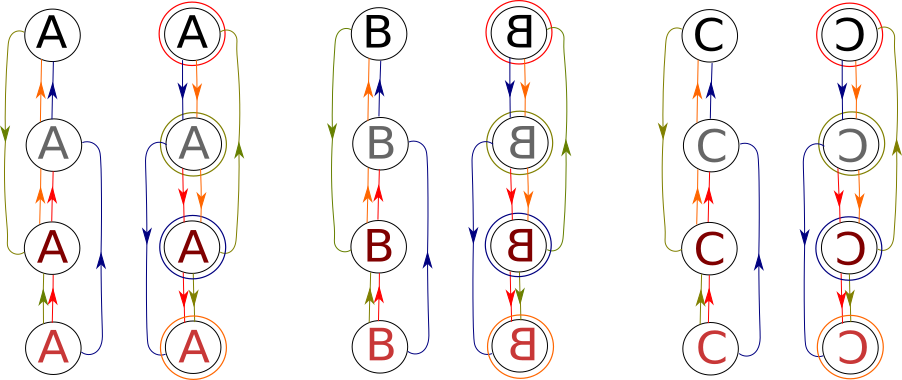}};
\end{tikzpicture}
\]
\caption{Curves obtained during the first step, when $g=3$ and $m=2$; one color per family of curves $\alpha^i$}
\label{first_step_examples}
\end{figure}
\begin{figure}[h!]
\[
\begin{tikzpicture}[scale=0.3]
\node (mypic) at (0,0) {\includegraphics[scale=0.3]{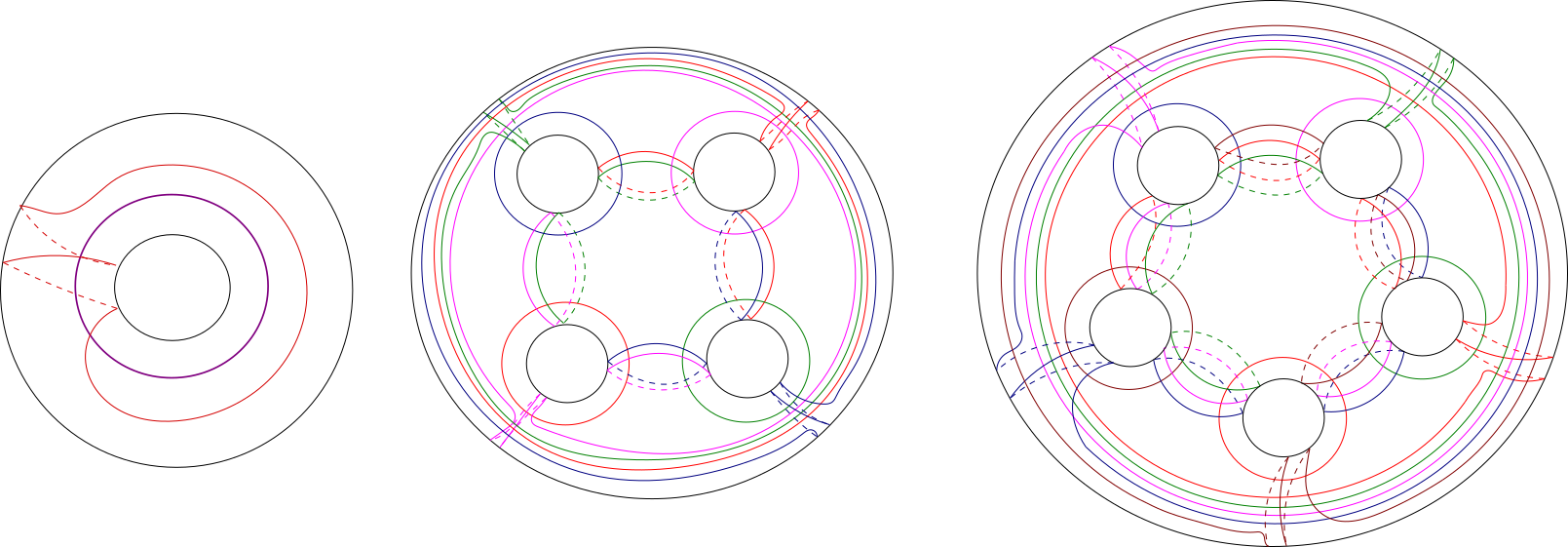}};
\end{tikzpicture}
\]
\caption{From left to right: a Heegaard diagram for the lens space $L(2,1)$; the corresponding $4$--section diagram for $\mathcal{S}_{2}(L(2,1))$; the corresponding $5$--section diagram for $\mathcal{S}_{3}(L(2,1))$}
\label{multisection_diagram_s2l21}
\end{figure}
\begin{proof}
In what follows, we keep the notations of Section \ref{secproof}. Notice that if $\lvert K \rvert = m+1$, $B_K$ is just one ray in $B^{m+1}$, therefore the notation $h_{B_K}^i$ makes sense for such a $K$. As the $3$--dimensional intersections are all built in the same fashion, we draw a system of curves for $Y_{\lbrace 1,...,m+1 \rbrace}=\mathcal{A}_{\lbrace 1,...,m+1 \rbrace} \cup \mathcal{B}_{\lbrace 1,...,m+1 \rbrace} \cup \mathcal{C}_{\lbrace 1,...,m+1 \rbrace}$, and obtain the other cut systems by rotation.\\
The central surface is the union of a $2g(m+2)$--punctured sphere (corresponding to $\mathcal{A}_{\lbrace 1,...,m+2 \rbrace}$) and $(m+2)g$ cylinders corresponding to $\mathcal{B}_{\lbrace 1,...,m+2 \rbrace } \cup \mathcal{C}_{\lbrace 1,...,m+2 \rbrace}$, i.e. each cylinder corresponds to one of the $g$ components of $\overline{\partial h_{B_{\lbrace 1,...,\hat{s},...,m+2 \rbrace}}^{s+1} \setminus D_\pm^{s+1}}$, for some $1 \leq s \leq m+2$  (see Figure \ref{central_surface}).
As $\mathcal{B}_{\lbrace 1,...,m+1 \rbrace}=h_{B_{\lbrace 1,...,m,m+2 \rbrace}}^{m+2}$, we can choose one meridian per cylinder in $\overline{\partial h_{B_{\lbrace 1,...,m,m+2 \rbrace }}^{m+2} \setminus D_\pm^{m+2}}$ as our first $g$ curves.
The next $mg$ curves are constructed as follows. Choose $mg$ disjoint, non--parallel disks in $H^o_{\delta, \lbrace 1,...,m+1 \rbrace}$, such that each disk merges two of the tubes inside one of the $g$ handles  $h_{\ell}$ at $B_{\lbrace 1,...,m+1 \rbrace}$ (more precisely, such a disk is obtained by choosing two tubes inside one $h_{\ell}$, two arcs on their boundaries, parallel to the core of each tube, and connecting these arcs by two disjoint arcs on the central copy of $S^2$; the curve thus obtained bounds a disk in $H^o_{\delta, \lbrace 1,...,m+1 \rbrace}$, see Figure \ref{diagram_step1}). The boundary of these disks does not lie totally on the central surface, as the boundary of a tube $h_{\ell,B_{\lbrace 1,...,m+1 \rbrace}}^{s+1}$ does not belong to the central surface if $s \neq m+2$. But we can extend these disks until their boun\-daries lie on the central surface, using the product structure in $\mathcal{C}_{\lbrace 1,...,m+1 \rbrace}$: for all $s \neq m+2$, the thickened cylinder $\partial h_{\ell, B_{\lbrace 1,...,\hat{s},...,m+1 \rbrace}}^{s+1}$ connects $\partial h_{\ell,B_{\lbrace 1,...,m+1 \rbrace}}^{s+1}$ and $\partial h_{\ell, B_{\lbrace 1,...,\hat{s},...,m+2 \rbrace}}^{s+1}$ (which belongs to the central surface) inside $\mathcal{C}_{\lbrace 1,...,m+1 \rbrace}$ (see Figure \ref{disk_extension}). See Figure \ref{first_step_example} and Figure \ref{first_step_examples} for examples showing these curves. 

\begin{figure}[h!]
\[
\begin{tikzpicture}[scale=0.5]
\node (mypic) at (0,0) {\includegraphics[scale=0.5]{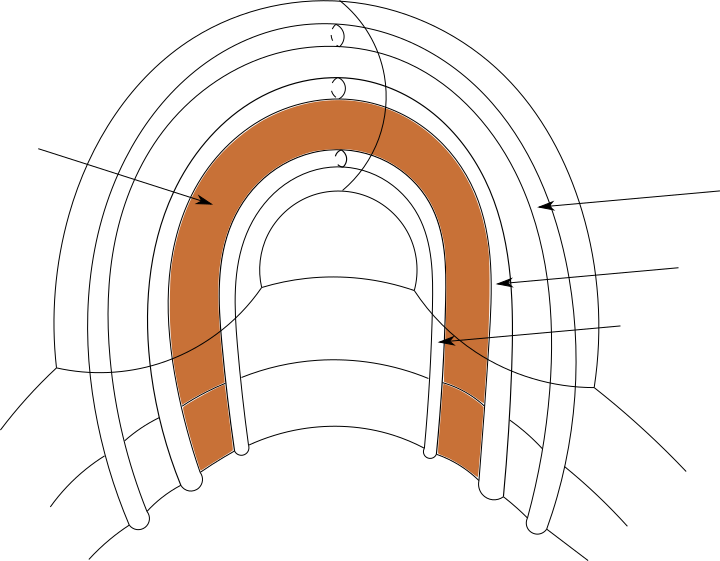}};
\node (a) at (8.8,-1.5) {$h_{1,B_{\lbrace 1,...,m+1 \rbrace}}^{s+1}$};
\node (a) at (10.4,0.1) {$h_{1,B_{\lbrace 1,...,m+1 \rbrace}}^{s+2}$};
\node (a) at (11.6,2.1) {$h_{1,B_{\lbrace 1,...,m+1 \rbrace}}^{m+1}$};
\node (a) at (4.4,1) {...};
\node (a) at (-6,1) {...};
\node (a) at (0,-5) {$S^2 \times \lbrace 0 \rbrace$};
\node (a) at (-9,3.5) {$D$};
\end{tikzpicture}
\]
\caption{A disk $D$ in brown, merging the tubes $h_{1,B_{\lbrace 1,...,m+1 \rbrace}}^{s+1}$ and $h_{1,B_{\lbrace 1,...,m+1 \rbrace}}^{s+2}$ inside the copy of the handle $h_{1}$ at $B_{\lbrace 1,...,m+1 \rbrace}$}
\label{diagram_step1}
\end{figure}
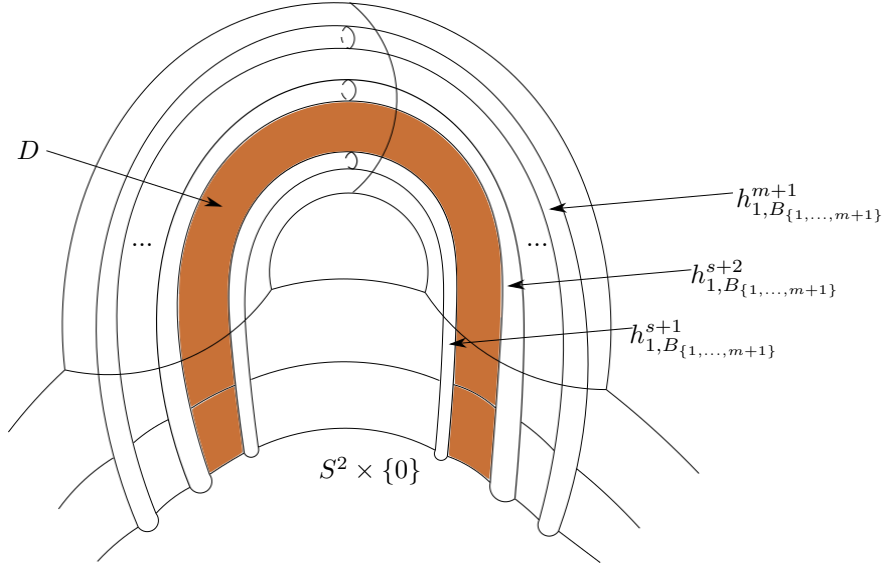

\begin{figure}[h!]
\[
\begin{tikzpicture}[scale=0.5]
\node (mypic) at (0,0) {\includegraphics[scale=0.5]{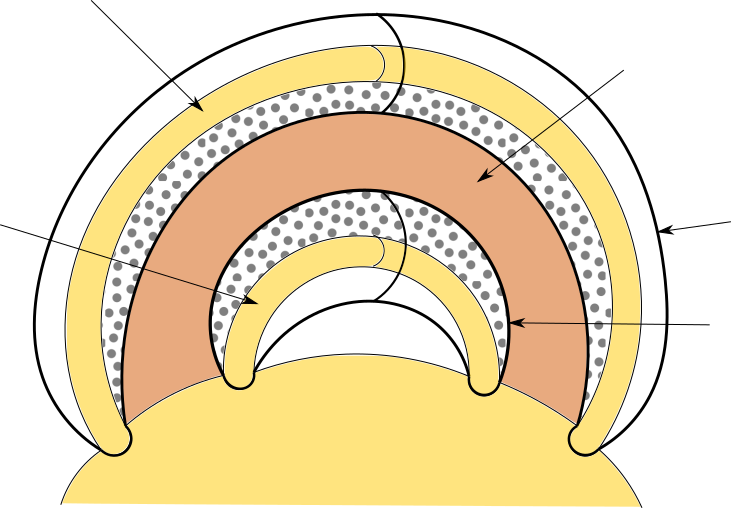}};
\node (a) at (12.1,1) {$h_{\ell,B_{\lbrace 1,...,m+1 \rbrace}}^{s+2}$};
\node (a) at (11.5,-2) {$h_{\ell,B_{\lbrace 1,...,m+1 \rbrace}}^{s+1}$};
\node (a) at (7,5.3) {$D$};
\node (a) at (-11,1.5) {$\partial h_{\ell,B_{\lbrace 1,...,\hat{s},...,m+2 \rbrace}}^{s+1}$};
\node (a) at (-9,7.5) {$\partial h_{\ell,B_{\lbrace 1,...,s,s+2,...,m+2 \rbrace}}^{s+2}$};

\node (a) at (0,-5) {$S^2 \times \lbrace 0 \rbrace$};
\end{tikzpicture}
\]
\caption{The disk $D$ (in brown) extends to the central surface (in yellow) through the dotted areas in $\mathcal{C}_{\lbrace 1,...,m+1 \rbrace}$\\{\footnotesize These dotted areas are two disks, one inside the thickened cylinder $\partial h_{\ell, B_{\lbrace 1,...,\hat{s},...,m+1 \rbrace}}^{s+1}$, which connects $\partial h_{\ell,B_{\lbrace 1,...,m+1 \rbrace}}^{s+1}$ to $\partial h_{\ell, B_{\lbrace 1,...,\hat{s},...,m+2 \rbrace}}^{s+1}$ inside $\mathcal{C}_{\lbrace 1,...,m+1 \rbrace}$, and the other inside the thickened cylinder $\partial h_{\ell, B_{\lbrace 1,...,s,s+2,...,m+1 \rbrace}}^{s+2}$, which connects $\partial h_{\ell,B_{\lbrace 1,...,m+1 \rbrace}}^{s+2}$ to $\partial h_{\ell, B_{\lbrace 1,...,s,s+2,...,m+2 \rbrace}}^{s+2}$ inside $\mathcal{C}_{\lbrace 1,...,m+1 \rbrace}$; these disks are pinched near $S^2 \times \lbrace 0 \rbrace$ because of the identification of the origin of the rays $r$ at $\lbrace 0 \rbrace \in B^{m+1}$  }} 
\label{disk_extension}
\end{figure}

We thus obtain $mg$ disjoint curves on the central surface, bounding disks in $Y_{\lbrace 1,...,m+1 \rbrace}$. After performing surgery along these curves and the $g$ curves constructed during the first step, we are down to a genus--$g$ surface $S_{\lbrace 1,...,m+1 \rbrace}$, which can be isotoped to the union of the (punctured) central copy of $S^2$ and the $g$ cylinders $\overline{\partial h_{B_{\lbrace 1,...,m+1 \rbrace}}^{1} \setminus D_{\pm}^{1}}$. This latter surface bounds $\mathcal{H}_{\epsilon,B_{\lbrace 1,...,m+1 \rbrace}}^1 \subset \mathcal{A}_{\lbrace 1,...,m+1 \rbrace}$. Because $H_{\epsilon,B_{\lbrace 1,...,m+1 \rbrace}}$ is a deformation retract of $\mathcal{H}_{\epsilon,B_{\lbrace 1,...,m+1 \rbrace}}^1$ inside $\mathcal{A}_{\lbrace 1,...,m+1 \rbrace}$, we can isotope the cut system for $H_{\epsilon, B_{\lbrace 1,...,m+1 \rbrace}}$ featured on the Heegaard surface of $M$ at $B_{\lbrace 1,...,m+1 \rbrace}$, until it reaches the central surface. This gives us $g$ curves, bounding disks in $\mathcal{A}_{m+1}$, and ensures that surgering $S_{\lbrace 1,...,m+1 \rbrace}$ along these curves produces a $2$--sphere.\\
Therefore we have defined a cut system of $(m+2)g$ curves on the central surface, corresponding to $Y_{\lbrace 1,...,m+1 \rbrace}$. The cut systems corresponding to the other $Y_{\lbrace 1,..., \hat{s},...,m+2 \rbrace}$ are obtained by the same process, with the curves obtained in the first step bounding meridian disks in $h_{B_{\lbrace 1,...,s-2,s,...,m+2 \rbrace}}^{s}$.\\
Notice that, because the handles $h_{\ell}$ of $H_{\delta}$ are parallel to the tubes $h_{\ell,B_{\lbrace 1,..., \hat{s},...m+2 \rbrace}}^{s+1}$, the curves corresponding to $H_{\delta}$ in the Heegaard diagram for $M$ must be isotopic to the meridians for the $h_{\ell,B_{\lbrace 1,..., \hat{s},...m+2 \rbrace}}^{s+1}$. This is why, in order to perform the last step, we need to start from a Heegaard diagram of $M$ such that $H_{\delta}$ is standardly embedded in $S^3$. This means that the curves corres\-ponding to $H_{\delta}$ are all meridians or all longitudes, as in Figure \ref{standard_diagram}. If the curves obtained in the first step are meridians (resp. longitudes), then we need to use a Heegaard diagram such that all curves corresponding to $H_{\delta}$ are meridians (resp. longitudes). With this requirement, the $g$ curves obtained during the last step are the curves defining $H_{\epsilon}$ on the Heegaard diagram. 
\end{proof}
A $4$--section diagram for $\mathcal{S}_{2}(L(2,1))$ and a $5$--section diagram for $\mathcal{S}_{3}(L(2,1))$ are featured on Figure \ref{multisection_diagram_s2l21}. One cut system of a $4$--section diagram for $\mathcal{S}_{2}(\mathbb{T}^3)$ is featured on Figure \ref{multisection_diagram_t3}. One cut system of a $4$--section diagram for $\mathcal{S}_{2}(\Sigma(2,3,5))$ is featured on Figure \ref{multisection_diagram_poincare} (with $\Sigma(2,3,5)$ the Poincaré homology sphere). As $\mathcal{S}_{m}(\Sigma(2,3,5))$ is a homology sphere of dimension $m+3$ (see Lemma \ref{homology_lemma}), this provides examples of multisected homology spheres in any dimension. 
\begin{Lemma}
\label{homology_lemma}
The homology groups of $\mathcal{S}_{m} (M)$ are, for $m > 1$:
\begin{equation*}
H_{k}(S_{m}(M)) =
\begin{cases}
\mathbb{Z} & \text{if $k=0, m+3$}\\
H_1(M) & \text{if $k=1, m+1$}\\ 
H_2(M) & \text{if $k=2, m+2$}\\
0  & \text{otherwise}
\end{cases}
\end{equation*}
The homology groups of $\mathcal{S}_{1} (M)$ are:
\begin{equation*}
H_{k}(S_{1}(M)) =
\begin{cases}
\mathbb{Z} & \text{if $k=0, 4$}\\
H_1(M) & \text{if $k=1$}\\ 
H_1(M) \oplus H_2(M) & \text{if $k=2$}\\
H_2(M) & \text{if $k=3$}\\
0  & \text{otherwise}
\end{cases}
\end{equation*}
\end{Lemma}
\begin{proof}
We have $H_{k}(M^o)=H_k(M)$, $0 \leq k \leq 2$, and $H_{k}(M^o)=0$, $k \geq 3$. Then according to the Künneth Theorem, we have, for $m>2$:
\begin{equation*}
H_{k}(M^o \times S^m) =
\begin{cases}
\mathbb{Z} & \text{if $k=0, m$}\\
H_1(M) & \text{if $k=1, m+1$}\\ 
H_2(M) & \text{if $k=2, m+2$}\\
0  & \text{otherwise}
\end{cases}
\end{equation*}
and also:
\begin{equation*}
H_{k}(S^2 \times S^m) =
\begin{cases}
\mathbb{Z} & \text{if $k=0, 2, m, m+2$}\\
0  & \text{otherwise}
\end{cases}
\end{equation*}
Then we can use the Mayer--Vietoris sequence to compute the homology of $\mathcal{S}_{m}(M)= (M^o \times S^m) \cup_{S^2 \times S^m} (S^2 \times B^{m+1})$:
\begin{equation*}
H_k(S^2 \times S^m) \rightarrow H_{k}(S^2) \oplus H_{k}(M^o \times S^m) \rightarrow H_{k}(S_{m}M) \rightarrow H_{k-1}(S^2 \times S^m)
\end{equation*} 
The case $m=1$ and $m=2$ can be dealt with in the same way, using:
\begin{equation*}
H_{k}(M^o \times S^1) =
\begin{cases}
\mathbb{Z} & \text{if $k=0$}\\
H_1(M) \oplus \mathbb{Z} & \text{if $k=1$}\\ 
H_1(M) \oplus H_2(M) & \text{if $k=2$}\\
H_2(M) & \text{if $k=3$}\\
0  & \text{otherwise}
\end{cases}
\end{equation*}
and:
\begin{equation*}
H_{k}(M^o \times S^2) =
\begin{cases}
\mathbb{Z} & \text{if $k=0$}\\
H_1(M) & \text{if $k=1,3$}\\ 
H_2(M) \oplus \mathbb{Z} & \text{if $k=2$}\\
H_2(M) & \text{if $k=4$}\\
0  & \text{otherwise}
\end{cases}
\end{equation*}
\end{proof}
\begin{figure}[h!]
\[
\begin{tikzpicture}[scale=0.3]
\node (mypic) at (0,0) {\includegraphics[scale=0.3]{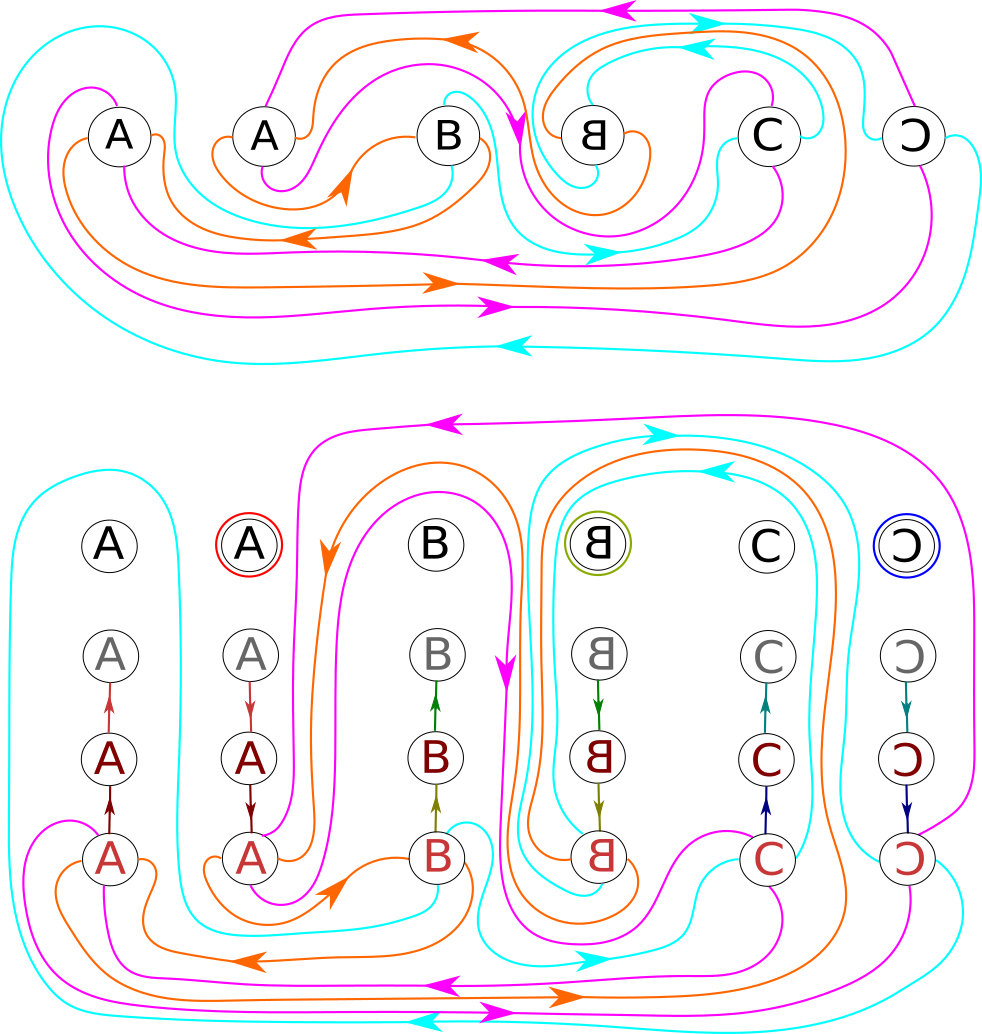}};
\end{tikzpicture}
\]
\caption{One cut system of a $4$--section diagram for $\mathcal{S}_2 (\mathbb{T}^3)$, based on the Heegaard diagram for $\mathbb{T}^3$ above, one color per curve}
\label{multisection_diagram_t3}
\end{figure}
\begin{figure}[h!]
\[
\begin{tikzpicture}[scale=0.2]
\node (mypic) at (0,0) {\includegraphics[scale=0.2]{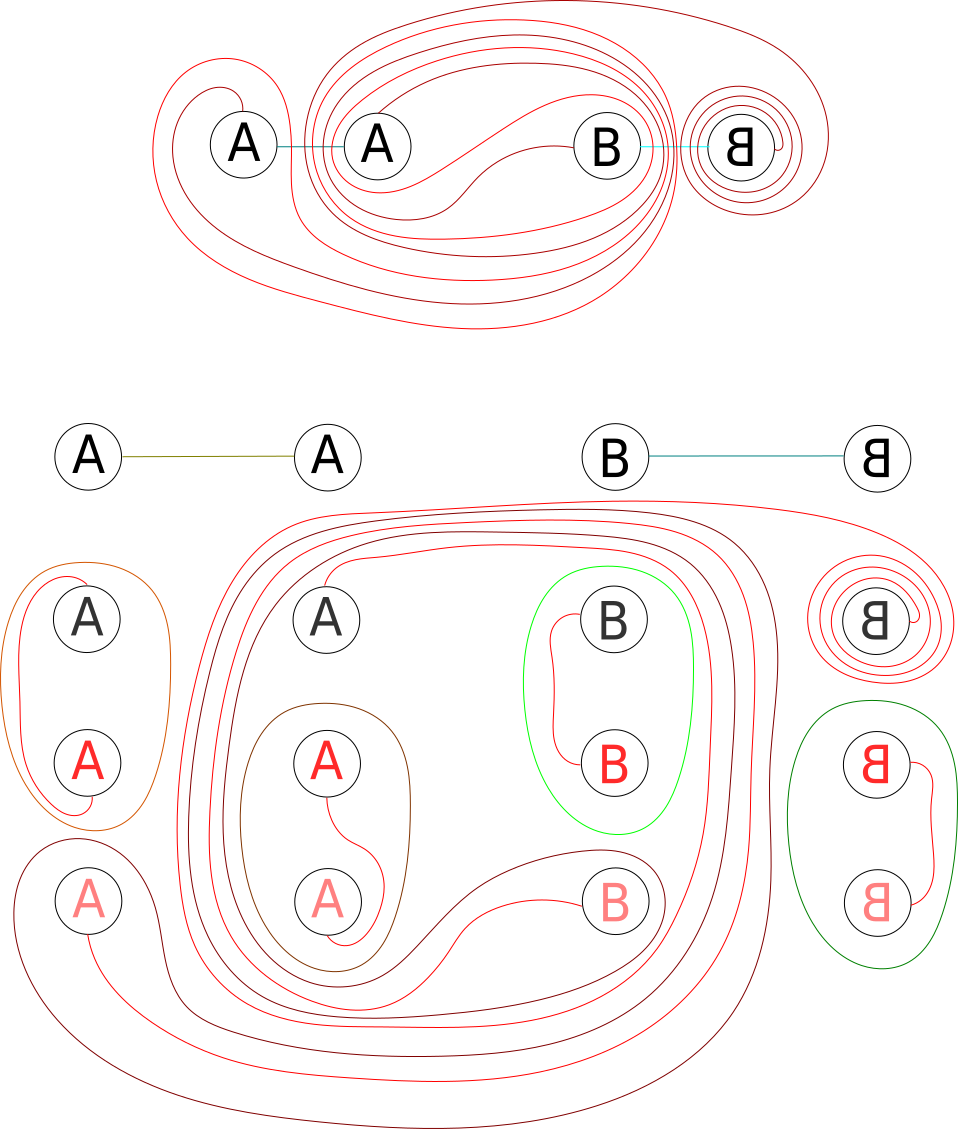}};
\end{tikzpicture}
\]
\caption{One cut system of a $4$--section diagram for $\mathcal{S}_2 (\Sigma(2,3,5))$, based on the Heegaard diagram for $\Sigma(2,3,5)$ above, one color per curve}
\label{multisection_diagram_poincare}
\end{figure}
\FloatBarrier

\bibliographystyle{/home/rudy/Documents/article/smfalpha.bst}
\bibliography{/home/rudy/Documents/Multisections_of_spun_manifolds/biblio_spun.bib}

\def\cprime{$'$}
\providecommand{\bysame}{\leavevmode\hbox to3em{\hrulefill}\thinspace}
\providecommand{\noopsort}[1]{}
\providecommand{\mr}[1]{\href{http://www.ams.org/mathscinet-getitem?mr=#1}{MR~#1}}
\providecommand{\zbl}[1]{\href{http://www.zentralblatt-math.org/zmath/en/search/?q=an:#1}{Zbl~#1}}
\providecommand{\jfm}[1]{\href{http://www.emis.de/cgi-bin/JFM-item?#1}{JFM~#1}}
\providecommand{\arxiv}[1]{\href{http://www.arxiv.org/abs/#1}{arXiv~#1}}
\providecommand{\doi}[1]{\url{https://doi.org/#1}}
\providecommand{\MR}{\relax\ifhmode\unskip\space\fi MR }
\providecommand{\MRhref}[2]{%
  \href{http://www.ams.org/mathscinet-getitem?mr=#1}{#2}
}
\providecommand{\href}[2]{#2}
\begin{thebibliography}{BACGM23}

\bibitem[BACGM23]{aribi2023multisections}
\bgroup\scshape{}F.~Ben~Aribi\egroup{}, \bgroup\scshape{}S.~Courte\egroup{},
  \bgroup\scshape{}M.~Golla\egroup{}, and
  \bgroup\scshape{}D.~Moussard\egroup{}, Multisections of higher-dimensional
  manifolds,  \emph{arXiv preprint arXiv:2303.08779} (2023).

\bibitem[GK16]{gay2016trisecting}
\bgroup\scshape{}D.~Gay\egroup{} and \bgroup\scshape{}R.~Kirby\egroup{},
  Trisecting 4--manifolds,  \emph{Geometry \& Topology} \textbf{20} no.~6
  (2016), 3097--3132.

\bibitem[Kos13]{kosinski2013differential}
\bgroup\scshape{}A.~A. Kosinski\egroup{}, \emph{Differential manifolds},
  Courier Corporation, 2013.

\bibitem[Mei18]{MR3917737}
\bgroup\scshape{}J.~Meier\egroup{}, Trisections and spun four-manifolds,
  \emph{Math. Res. Lett.} \textbf{25} no.~5 (2018), 1497--1524.

\bibitem[Mou23]{moussard2023multisections}
\bgroup\scshape{}D.~Moussard\egroup{}, Multisections of surface bundles and
  bundles over the circle,  \emph{arXiv preprint arXiv:2305.05619} (2023), to
  appear in Publicacions Matemàtiques.

\bibitem[RT18]{rubinstein2018generalized}
\bgroup\scshape{}J.~H. Rubinstein\egroup{} and
  \bgroup\scshape{}S.~Tillmann\egroup{}, Generalized trisections in all
  dimensions,  \emph{Proceedings of the National Academy of Sciences}
  \textbf{115} no.~43 (2018), 10908--10913.

\end{thebibliography}
\end{document}